\documentclass[12pt]{amsart}
\usepackage[utf8]{inputenc}
\usepackage{amsfonts}
\usepackage{amsmath}
\usepackage{amssymb}
\usepackage{amsthm}
\usepackage{xcolor}
\usepackage[margin=1.3in]{geometry}
\usepackage{graphicx}
\usepackage[11pt]{moresize}
\usepackage{tikz}
\tikzstyle{vertex}=[circle,draw=black,fill=black,inner sep=0,minimum size=3pt,text=white,font=\footnotesize]

\newtheorem{theorem}{Theorem}
\newtheorem*{conjecture*}{Conjecture}
\newtheorem{proposition}{Proposition}[section]
\newtheorem{lemma}[proposition]{Lemma}
\newtheorem{corollary}[proposition]{Corollary}
\theoremstyle{remark}
\newtheorem{definition}[proposition]{Definition}

\newtheorem*{remark*}{Remark}

\newcommand{\vs}{\vspace{3mm}}

\newcommand{\hs}{\hspace{1mm}}
\newcommand{\hsss}{\hspace{3mm}}

\newcommand{\Z}{\mathbb{Z}}

\newcommand{\N}{\mathbb{N}}

\newcommand{\mc}{\mathcal}

\newcommand{\ep}{\epsilon}

\newcommand{\sub}{\subseteq}

\newcommand{\modd}[1]{\text{ mod } #1}
\newcommand\height{5}
\newcommand\heighty{5.3}

\newcommand{\ol}{\overline}
\newcommand{\wt}{\widetilde}

\title{A Random Analogue of Gilbreath's Conjecture}
\author{Zachary Chase}
\thanks{The author is partially supported by Ben Green's Simons Investigator Grant 376201 and gratefully acknowledges the support of the Simons Foundation.}
\address{Mathematical Institute, Andrew Wiles Building, Radcliffe Observatory Quarter, Woodstock Road, Oxford OX2 6GG, UK}
\email{zachary.chase@maths.ox.ac.uk}
\date{May 1, 2020}

\begin{document}

\begin{abstract}
A well-known conjecture of Gilbreath, and independently Proth from the 1800s, states that if $a_{0,n} = p_n$ denotes the $n^{\text{th}}$ prime number and $a_{i,n} = |a_{i-1,n}-a_{i-1,n+1}|$ for $i, n \ge 1$, then $a_{i,1} = 1$ for all $i \ge 1$. It has been postulated repeatedly that the property of having $a_{i,1} = 1$ for $i$ large enough should hold for any choice of initial $(a_{0,n})_{n \ge 1}$ provided that the gaps $a_{0,n+1}-a_{0,n}$ are not too large and are sufficiently random. We prove (a precise form of) this postulate.   
\end{abstract}

\maketitle

\section{Introduction}

Given any sequence of non-negative integers $(a_n)_{n \ge 1}$, we can form the sequence of non-negative integers $(|a_n-a_{n+1}|)_{n \ge 1}$. Start with the primes as the initial sequence and iterate this consecutive differencing procedure. Gilbreath's conjecture is that the first term in every sequence, starting with the first iteration, is a $1$. Precisely, if $a_{0,n} = p_n$ for $n \ge 1$ and $a_{i,n} = |a_{i-1,n}-a_{i-1,n+1}|$ for $i, n \ge 1$, then $a_{i,1} = 1$ for all $i \ge 1$. Below are the first few terms of the first few iterations. 

\vspace{3.1mm}

\hspace{49mm} 2 \hsss 3 \hsss 5 \hsss 7 \hspace{1mm} 11 \hspace{-.5mm} 13 \hspace{-.1mm} 17 

\begin{center}
1 \hsss 2 \hsss 2 \hsss 4 \hsss 2 \hsss 4

1 \hsss 0 \hsss 2 \hsss 2 \hsss 2

1 \hsss 2 \hsss 0 \hsss 0

1 \hsss 2 \hsss 0

1 \hsss 2

1
\end{center}

\vspace{2.6mm}

Proth \cite{proth2} discussed Gilbreath's conjecture in 1878, before Gilbreath independently made the conjecture. Many sources claim Proth asserted he had a proof of the conjecture, and that his proof was wrong. However, we believe this claim is baseless. See Section \ref{history} for more details. Odlyzko \cite{odlyzko} verified Gilbreath's conjecture for $1 \le i \le \pi(10^{13}) \approx 3.34\times 10^{11}$. One is led to wonder how special the primes are in Gilbreath's conjecture and whether any sequence beginning with $2$ followed by an increasing sequence of odd numbers with small and ``random" gaps between them will have first term $1$ from some iteration onwards. 

\vs

Odlyzko, at the end of Section $2$ of \cite{odlyzko}, speculates that such a random sequence indeed will have first term $1$ from some iteration onwards. Additionally, Problem 68 of \cite{10lectures} asks what gap or density properties of an initial sequence suffices to ensure the conclusion of Gilbreath's conjecture. Despite Gilbreath's conjecture being around for over a decade and several additional sources postulating that the conjecture should hold for initial sequences with small and random gaps, as of date, nothing has actually been \textit{proven} along these lines, nor about Gilbreath's conjecture specifically. 

\vspace{1.5mm}

In this paper, we initiate a rigorous study of Gilbreath's conjecture by proving a random analogue of it.  

\begin{theorem}\label{truemain}
Let $f : \N \to \N$ be an increasing function with $f(M) \le \frac{1}{100}\frac{\log\log M}{\log\log\log M}$ for $M$ large and $f(M) \ge 2$ for all $M \ge 1$. Let $a_1,a_2,\dots$ be a random infinite sequence formed as follows. Let $a_1 = 2, a_2 = 3$, and for $n \ge 2$, $a_{n+1} = a_n+2u_n$, where $u_n$ is drawn uniformly at random from $\{0,1,\dots,f(n)-1\}$, independent of the other $u_i$'s. Then, with probability $1$, there is some $M_0$ so that for all $M \ge M_0$, after $M$ iterations of consecutive differencing, the first term of the sequence is a $1$.
\end{theorem}

Computations suggest that Gilbreath's conjecture holds because $0$s and $2$s form to the right of the leading $1$ early on. We prove Theorem \ref{truemain} by showing that our random initial sequence indeed has that property almost surely. Since the first iteration is $1, 2u_2, 2u_3, \dots$, if we ignore the leading $1$ and divide by $2$, what we wish to show is encapsulated by the following theorem, which is the heart of the paper.

\begin{theorem}\label{main}
For $M$ large, for any $C$ with $2 \le C \le \frac{1}{100}\frac{\log\log M}{\log\log\log M}$, if we form an initial sequence of length $M$ by choosing numbers from $\{0,\dots,C-1\}$ independently and uniformly at random, then, with probability at least $1-e^{-e^{\sqrt[20]{\log M}}}$, after $e^{\sqrt[5]{\log M}}$ iterations of consecutive differencing, everything is a $0$ or $1$. 
\end{theorem}

The randomness in Theorem \ref{main} is certainly necessary. For example, if the initial sequence consists of only $0$s and $3$s, then after any number of iterations, everything is still a $0$ or $3$. However, there are more exotic examples of initial sequences

\vspace{0.5mm}

\begin{center}
2 \hsss 0 \hsss 6 \hsss 0 \hsss 2 \hsss 2 \hsss 6 \hsss 5 \hsss 0 \hsss 0 \hsss 6 \hsss 1 \hsss 3 \hsss 2 \hsss 2 \hsss 3 \hsss 0 \hsss 6 \hsss 0 \hsss 5 \hsss 

2 \hsss 6 \hsss 6 \hsss 2 \hsss 0 \hsss 4 \hsss 1 \hsss 5 \hsss 0 \hsss 6 \hsss 5 \hsss 2 \hsss 1 \hsss 0 \hsss 1 \hsss 3 \hsss 6 \hsss 6 \hsss 5 \hsss 

4 \hsss 0 \hsss 4 \hsss 2 \hsss 4 \hsss 3 \hsss 4 \hsss 5 \hsss 6 \hsss 1 \hsss 3 \hsss 1 \hsss 1 \hsss 1 \hsss 2 \hsss 3 \hsss 0 \hsss 1 \hsss 

4 \hsss 4 \hsss 2 \hsss 2 \hsss 1 \hsss 1 \hsss 1 \hsss 1 \hsss 5 \hsss 2 \hsss 2 \hsss 0 \hsss 0 \hsss 1 \hsss 1 \hsss 3 \hsss 1 \hsss 

0 \hsss 2 \hsss 0 \hsss 1 \hsss 0 \hsss 0 \hsss 0 \hsss 4 \hsss 3 \hsss 0 \hsss 2 \hsss 0 \hsss 1 \hsss 0 \hsss 2 \hsss 2 \hsss 

2 \hsss 2 \hsss 1 \hsss 1 \hsss 0 \hsss 0 \hsss 4 \hsss 1 \hsss 3 \hsss 2 \hsss 2 \hsss 1 \hsss 1 \hsss 2 \hsss 0 \hsss 

0 \hsss 1 \hsss 0 \hsss 1 \hsss 0 \hsss 4 \hsss 3 \hsss 2 \hsss 1 \hsss 0 \hsss 1 \hsss 0 \hsss 1 \hsss 2 \hsss 

1 \hsss 1 \hsss 1 \hsss 1 \hsss 4 \hsss 1 \hsss 1 \hsss 1 \hsss 1 \hsss 1 \hsss 1 \hsss 1 \hsss 1 \hsss 

0 \hsss 0 \hsss 0 \hsss 3 \hsss 3 \hsss 0 \hsss 0 \hsss 0 \hsss 0 \hsss 0 \hsss 0 \hsss 0 \hsss 
\end{center}

\noindent for which all future iterations have only $0$s and $3$s (say). These exotic examples\footnote{To clarify, in the setting in which the primes are the initial sequence, the analogous situation to having only $0$s and $3$s is having only $0$s and $6$s past the first index, making the first index very likely to repeatedly change from $1$ to $5$ (see Lemma \ref{mod2}), thereby violating Gilbreath's conjecture.} suggest that we are far away from a proof of Gilbreath's conjecture. 

\vs

\section{A General Bootstrapping Argument}

In this section, we prove a result about random walks on regular directed graphs that will be of use to proving Theorem \ref{main}. 

\begin{definition}
A directed graph is \textit{regular} if there is a positive integer $d$ such that each vertex has in-degree and out-degree equal to $d$. We allow our graphs to have self-loops (but no multiple edges). For our discussion, a \textit{simple random walk} on a regular directed graph of degree $d$ is formed by choosing a starting point uniformly at random, and then walking along the directed edges, with each out-edge chosen with probability $1/d$, independent of the previous steps.  
\end{definition}

\begin{proposition}\label{bootstrap}
Let $G = (V,E)$ be a regular directed graph. Suppose $V$ is red-blue colored such that the probability a simple random walk on $G$ of length $L$ consists entirely of red vertices is at least $c$. Then the probability a simple random walk on $G$ of length $\lfloor(1+\frac{1}{10}c^2)L\rfloor$ consists entirely of red vertices is at least $\frac{1}{10}c^2$.
\end{proposition}

\begin{proof}
Let $X_1,X_2,\dots$ denote the steps of a simple random walk. Define functions $w_1,\dots,w_L$ on $V$ by $w_j(v) := \Pr(X_1,\dots,X_L \text{ all red} | X_j = v).$ Note (by, e.g., induction on the number of steps) the regularity assumption implies $$w_j(v) = |V|\Pr(X_1,\dots,X_L \text{ all red}, X_j = v).$$ Thus, for any $j$, letting $$w_j(V) := \sum_{v \in V} w_j(v),$$ we have by assumption \begin{align*} w_j(V) &= \sum_v |V|\Pr(X_1,\dots,X_L \text{ all red}, X_j = v) \\ &= |V|\Pr(X_1,\dots,X_L \text{ all red}) \\ &\ge c|V|.\end{align*} Let $K = \lceil \frac{3}{c^2} \rceil$, and let $k_1,\dots,k_K$ be $k_j := \lfloor\frac{j}{K}L\rfloor$. By Cauchy-Schwarz, \begin{align}\label{csin} \left(\sum_v \sum_j w_{k_j}(v)\right)^2 &\le \left[\sum_v 1^2\right]\cdot\left[\sum_v \left(\sum_j w_{k_j}(v)\right)^2\right] \\ \nonumber &= |V|\left[\sum_j \sum_v w_{k_j}(v)^2+2\sum_{j < j'} \sum_v w_{k_j}(v)w_{k_{j'}}(v)\right].\end{align} Note, since $||w_j||_\infty \le 1$, we have $$\sum_j \sum_v w_{k_j}(v)^2 \le \sum_j \sum_v w_{k_j}(v) = \sum_j |V|\Pr(X_1,\dots,X_L \text{ all red}) \le K|V|;$$ also, $$\sum_v \sum_j w_{k_j}(v) = \sum_j w_{k_j}(V) \ge Kc|V|.$$ So \eqref{csin} implies $$K^2c^2|V|^2 \le |V|\left[K|V|+2\sum_{j < j'} \sum_v w_{k_j}(v)w_{k_{j'}}(v)\right],$$ and thus, since $K^2c^2|V|-K|V|$ is increasing in $K$ for $K \ge 3/c^2$, $$\frac{6}{c^2}|V| \le 2\sum_{j < j'} \sum_v w_{k_j}(v)w_{k_{j'}}(v).$$ By the pigeonhole principle, there are $j < j'$ with $$\sum_v w_{k_j}(v)w_{k_{j'}}(v) \ge \frac{1}{K^2}\frac{3}{c^2}|V|.$$ Using {\small $$w_{k_j}(v) \le \Pr(X_{k_j+1},\dots,X_L \text{ all red} | X_{k_j} = v) = \Pr(X_{k_{j'}+1},\dots,X_{L+k_{j'}-k_j} \text{ all red} | X_{k_{j'}} = v),$$} which is true merely due to translation invariance of the random walk, and $$w_{k_{j'}}(v) \le \Pr(X_1,\dots,X_{k_{j'}} \text{ all red} | X_{k_{j'}} = v),$$ we obtain {\footnotesize \begin{align*} \frac{1}{K^2}\frac{3}{c^2}|V| &\le \sum_v \Pr(X_1,\dots,X_{k_{j'}} \text{ all red} | X_{k_{j'}} = v)\Pr(X_{k_{j'}+1},\dots,X_{L+k_{j'}-k_j} \text{ all red} | X_{k_{j'}} = v) \\ &= |V| \sum_v \Pr(X_1,\dots,X_{k_{j'}} \text{ all red}, X_{k_{j'}} = v)\Pr(X_{k_{j'}+1},\dots,X_{L+k_{j'}-k_j} \text{ all red} | X_{k_{j'}} = v) \\ &= |V|\sum_v \Pr(X_1,\dots,X_{L+k_{j'}-k_j} \text{ all red}, X_{k_{j'}} = v)\\ &= |V|\Pr(X_1,\dots,X_{L+k_{j'}-k_j} \text{ all red}),\end{align*}} yielding $$\Pr(X_1,\dots,X_{L+k_{j'}-k_j} \text{ all red}) \ge \frac{1}{K^2}\frac{3}{c^2}.$$ Note $K \le \frac{3}{c^2}+1 \le \frac{4}{c^2}$, so $\frac{1}{K^2}\frac{3}{c^2} \ge \frac{3}{16}c^2 \ge \frac{1}{10}c^2$. Since the proposition is trivial if $L < 10/c^2$, we may assume $L \ge 10/c^2$ to obtain $k_{j'}-k_j \ge \frac{L}{K}-1 \ge \frac{c^2}{4}L-1 \ge \frac{c^2}{10}L$. 
\end{proof}

\vspace{1mm}

\begin{remark*}
It is natural to think that Proposition \ref{bootstrap} can be extended, in some form, to arbitrary length increases. However, such an extension is not possible in general (note that iterating Proposition \ref{bootstrap} results in only a summable geometric series of length increases). For example, consider $V = \{1,\dots,n\}, E = \{(1 \mapsto 2),\dots,(n-1 \mapsto n),(n \mapsto 1)\}$ with the vertices $\{1,\dots,\frac{1}{10}n\}$ colored red and the rest blue. Then with $L = \frac{1}{20}n$ and $c = \frac{1}{20}$, it holds that a simple random walk on $G$ of length $L$ will hit only red vertices with probability at least $c$. However, of course no simple (random) walk on $G$ of length $5L = \frac{1}{2}n$ will hit only red vertices.

\vspace{1.5mm}

Examples of such ``bad" colorings also exist on the graph we apply Proposition \ref{bootstrap} to, namely a Debrujin graph. We don't think these colorings are actually the ones we need to address in our proof of Theorem \ref{main}, but we couldn't prove that.
\end{remark*}

\section{A Lower Bound for Ending with $0$}

We begin by exploiting the main property of the ``dynamical system" of taking consecutive differences: the supremum never increases. In fact, we use that it quickly decreases provided there is no trivial obstruction to it doing so (Lemma \ref{blockdestruction}). 

\begin{definition}
We say non-negative integers $a_1,\dots,a_i$ \textit{come from} $\wt{a}_1,\dots,\wt{a}_{i+1}$ if $|\wt{a}_j-\wt{a}_{j+1}| = a_j$ for each $1 \le j \le i$. Given $a_1,\dots,a_i$ and a subset $E \sub \Z$, an \textit{$E$-block} is a contiguous set of terms $a_{j_1+1},\dots,a_{j_1'}$ such that $a_j \in E$ for each $j_1+1 \le j \le j_1'$; the \textit{length} of the block is $j_1'-j_1$.  
\end{definition}

\begin{lemma}\label{blockdestruction}
Let $a_1,\dots,a_i$ be non-negative integers with $d := \max_j a_j$. Let $L$ denote the length of the longest $\{0,d\}$-block containing at least one $d$. If $L \le i-1$, then, after $L$ iterations of consecutive differencing, the largest number is at most $d-1$.
\end{lemma}

\begin{proof}
We induct on $L$. For $L=1$, the result is clear. Assume $L \ge 2$ and the result is true for all $L' < L$. It is easy to see that, since $d$ is the maximum, any $\{0,d\}$-block containing a $d$ after an iteration would have had to have come from a $\{0,d\}$-block of greater length containing a $d$, so the longest $\{0,d\}$-block containing a $d$ after one iteration is at most $L-1$, say $L'$. By induction, after $L'$ more iterations, the largest number is at most $d-1$. It follows that after $L$ (total) iterations, the largest number is at most $d-1$. 
\end{proof}

\vspace{1mm}

So, to prove Theorem \ref{main}, ``all" we need to do is argue that long $\{0,d\}$-blocks are unlikely to exist. In this next lemma, we observe that any large $\{0,d\}$-block essentially must have come from a block with no $0$s. 

\vspace{1mm}

\begin{lemma}\label{inverseiterates}
Suppose that after $i$ iterations, there is a $d\Z$-block of length $L$. Then either there was a $d\Z$-block of length $L+i$ in the initial sequence, or there is some $i'$, $0 \le i' \le i-1$, such that after $i'$ iterations, there is a block of length $L+i-i'$ with no $0$s.
\end{lemma}

\begin{proof}
We prove by induction on $i$ the statement for all $L$. For $i=0$, the result is tautological. Take $i \ge 1$, and suppose the result holds for $i-1$. The $d\Z$-block of length $L$ had to come from either a $d\Z$-block of length $L+1$ or a block of length $L+1$ with no $0$s (since everything will have the same residue modulo $d$), so we are done by the induction hypothesis. 
\end{proof}

Another nice property of the consecutive differencing operation is that it ``commutes" with reducing mod $2$. This allows for a decently explicit formula for the parity of a term after a given number of iterations, merely in terms of the parities of the initial terms. 

\begin{definition}
For non-negative integers $a_1,a_2$, define $f_1(a_1,a_2) = |a_1-a_2|$, and for any $i \ge 2$ and non-negative $a_1,\dots,a_{i+1}$, define $f_i(a_1,\dots,a_{i+1}) = |f_{i-1}(a_1,\dots,a_i)-f_{i-1}(a_2,\dots,a_{i+1})|$. We say $a_1,\dots,a_{i+1}$ \textit{ultimately iterate} to $f_i(a_1,\dots,a_{i+1})$. 
\end{definition}

\begin{lemma}\label{mod2}
For any $i \ge 1$, there is a subset $J_i \sub [i+1]$ containing $1$ and $i+1$ so that for any non-negative integers $a_1,\dots,a_{i+1}$, $f_i(a_1,\dots,a_{i+1}) \equiv \sum_{j \in J_i} a_j \modd{2}$.
\end{lemma}

\begin{proof}
We induct on $i$. For $i=1$, the result follows from $|a_1-a_2| \equiv a_1+a_2 \modd 2$. Assume $i \ge 2$ and the result is true for $i-1$. Note that $f_i(a_1,\dots,a_{i+1}) \equiv |f_{i-1}(a_1,\dots,a_i)-f_{i-1}(a_2,\dots,a_{i+1})| \equiv f_{i-1}(a_1,\dots,a_i)+f_{i-1}(a_2,\dots,a_{i+1}) \equiv$

\noindent $\sum_{j \in J_{i-1}} a_j+\sum_{j \in J_{i-1}} a_{j+1} \equiv \sum_{j \in J_{i-1} \triangle (J_{i-1}+1)} a_j \modd{2}$. By induction, $J_{i-1}$ contains $1$ and $i$, and so $J_i := J_{i-1} \triangle (J_{i-1}+1)$ contains $1$ and $i+1$, as desired.
\end{proof}

\vspace{1.5mm}

We take a moment to note some useful corollaries of Lemma \ref{mod2} which tells us that the parity of what $a_1,\dots,a_{i+1}$ ultimately iterate to depends linearly on each of the parities of $a_1$ and $a_{i+1}$. For example, let $a_1,\dots,a_{i+1}$ be drawn independently, uniformly at random from $\{0,\dots,C-1\}$. Then, the probability $a_1,\dots,a_{i+1}$ ultimately iterate to an even integer is between $\frac{1}{3}$ and $\frac{2}{3}$. And the probability that, for $j = i/2$ say, all of $f_j(a_t,\dots,a_{t+j})$ are even, for $t=1,\dots,i/2$, is exponentially small in $i/2$.  

\vs

Let $[C]_0 = \{0,\dots,C-1\}$. 

\vs

The following proposition shows that $0$s are not too rare, which will be useful in conjuction with Lemma \ref{inverseiterates}. Before the proof, we introduce some notation (for a given $C$ and $i$). Define $i_0 = i$ and $i_{j+1} = \lfloor \frac{i_j}{100C^2} \rfloor$ for $0 \le j \le C-3$. For $1 \le j \le C-2$, let $E_j$ denote the event that after $i-i_{j-1}$ iterations there's a $\{0,C-j\}$-block of length (at least) $i_{j-1}-i_j$. For example, $E_1$ is the event that after $0$ iterations, there's a $\{0,C-1\}$-block of length $i-i_1$, and $E_2$ is the event that after $i-i_1$ iterations, there's a $\{0,C-2\}$-block of length $i_1-i_2$.   

\begin{proposition}\label{lowerbound}
For any $C \ge 2$ and any $i \ge (200C^2)^{2C}$, if $a_1,\dots,a_i$ are chosen independently and uniformly at random from $\{0,\dots,C-1\}$, then the probability they ultimately iterate to $0$ is at least $\frac{1}{200C^2}$. 
\end{proposition}

\begin{proof}
Fix $C \ge 2$ and $i \ge (200C^2)^{2C}$. If $C=2$, then Lemma \ref{mod2} gives the result, so assume $C \ge 3$. We may suppose that the desired probability is at most $0.01$. Let $\mc{B}_0$ denote all $i$-tuples in $[C]_0^i$ that ultimately iterate to something $0$ mod $2$; we say ``conditional probability" when speaking of the conditional probability that $\mc{B}_0$ induces. Then, by Lemma \ref{mod2}, the conditional probability of ultimately iterating to $0$ is at most $0.03$, and so the conditional probability of not having only $0$s and $1$s after some iteration is at least $0.97$. 

\vs

Therefore, with conditional probability at least $0.97$, some $E_j$ occurs. Indeed, otherwise, repeated use of Lemma \ref{blockdestruction} shows that after $i-i_{C-2}$ iterations, everything is a $0$ or a $1$: after $i-i_1$ iterations, there are no more $(C-1)$s and thus no $(C-1)$s ever again; after $i-i_2$ iterations, there are no more $(C-2)$s and thus no $(C-2)$s ever again, etc.. 

\vs

Therefore, by the pigeonhole principle, there is some $j$, $1 \le j \le C-2$, such that $E_j$ occurs with conditional probability at least $\frac{0.97}{C-2}$. Clearly $j$ cannot be $1$, since we have the uniform distribution after $0$ iterations. Also, $j$ must be such that $C-j$ is odd, since by Lemma \ref{mod2}, the probability of having $2i_j$ evens in a row is at most $(\frac{2}{3})^{2i_j} \le (\frac{2}{3})^{2(200C^2)^C}$ (since, as is easy to verify, $i_j \ge i_{C-2} \ge (200C^2)^C$ for each $j$). Since after $i-i_{j-1}$ iterations, there are only $i_{j-1}$ indices, a block of length $i_{j-1}-i_j$ must contain the block $[i_j+1,i_{j-1}-i_j]$ (see figure $1$). So, with conditional probability at least $\frac{0.97}{C-2}$, all indices $i_j+\Delta$, for $1 \le \Delta \le i_{j-1}-2i_j$, will be a $0$ or $C-j$. 

\vs

Let $a_1,\dots,a_i$ be the initial sequence, and note that, after $i-i_{j-1}$ iterations, none of the indices $i_j+\Delta$ depend on $a_1$ or $a_i$ (only the first and last indices do). Therefore, by Lemma \ref{mod2}, with (unconditional) probability at least $\frac{0.30}{C-2}$, all $i_j+\Delta$ will be $0$ or $C-j$. Now, note that after $\ol{i} := i-i_{j-1}$ iterations, the integer at any index $r$ is equal to $f_{\ol{i}}(a_r, a_{r+1},\dots,a_{r+\ol{i}})$. 

\vspace{-4.8mm}

\begin{center} 
\begin{tikzpicture}

\node[vertex] (0) at (0,\height) {};
\node[vertex] (ej) at (0.4,\height) {};
\node[vertex] (1-ej) at (2.1,\height) {};
\node[vertex] (1) at (2.5,\height) {};
\node[vertex] (i-1) at (7.5,\height) {};
\node[vertex] (ej+i-1) at (7.9,\height) {};
\node[vertex] (1-ej+i-1) at (9.6,\height) {};
\node[vertex] (1+i-1) at (10,\height) {};
\draw[thick] (0) -- (1+i-1);

\node[vertex] (0') at (3.75,.3) {};
\node[vertex] (ej') at (4.15,.3) {};
\node[vertex] (1-ej') at (5.85,.3) {};
\node[vertex] (1') at (6.25,.3) {};
\draw[thick] (0') -- (1');
\draw[line width=.5mm, red] (ej') -- (1-ej');

\draw[thick] (0') -- (0);
\draw[thick] (0') -- (i-1);
\draw[thick] (ej') -- (ej);
\draw[thick] (ej') -- (ej+i-1);
\draw[thick] (1-ej') -- (1-ej);
\draw[thick] (1-ej') -- (1-ej+i-1);
\draw[thick] (1') -- (1);
\draw[thick] (1') -- (1+i-1);

\node at (0,\heighty) {{\tiny $0$}};
\node at (0.5,\heighty) {{\tiny $i_j$}};
\node at (1.7,\heighty) {{\tiny $i_{j-1}-i_j$}};
\node at (2.9,\heighty) {{\tiny $i_{j-1}$}};
\node at (7, \heighty) {{\tiny $i-i_{j-1}$}};
\node at (8.2, 5.7) {{\tiny $i-(i_{j-1}-i_j)$}};
\node at (9.3, \heighty) {{\tiny $i-i_j$}};
\node at (10, \heighty) {{\tiny $i$}};

\node at (3.75,0) {{\tiny $0$}};
\node at (4.25,0) {{\tiny $i_j$}};
\node at (5.5,0) {{\tiny $i_{j-1}-i_j$}};
\node at (6.6,0) {{\tiny $i_{j-1}$}};

\draw[thick] (ej+i-1) -- (7.9,5.5);

\end{tikzpicture}

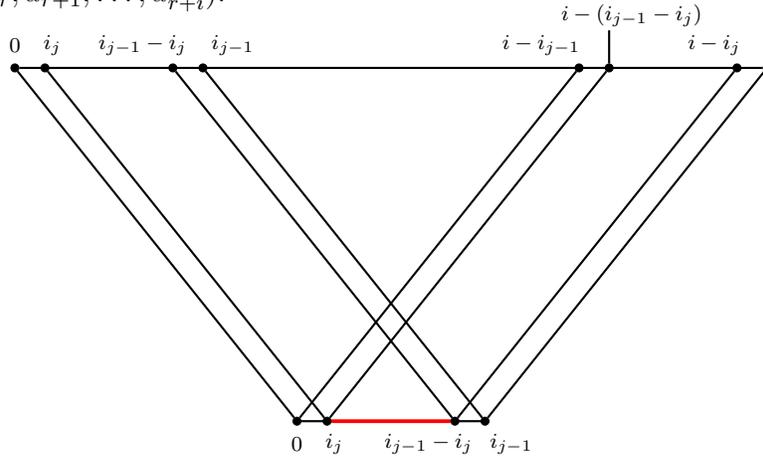
{\tiny Figure 1: Indicates which initial indices (in $[i]$) a particular index after $\ol{i}$ iterations depends on.}
\end{center}

Define a (regular) directed graph on $[C]_0^{\ol{i}}$ by $(x_1,\dots,x_{\ol{i}}) \to (x_2,\dots,x_{\ol{i}},y)$ for any $x_1,\dots,x_{\ol{i}},y \in [C]_0$. Color a tuple $(x_1,\dots,x_{\ol{i}}) \in [C]_0^{\ol{i}}$ ``red" if and only if it ultimately iterates to $0$ or $C-j$. The fact that, with probability at least $\frac{0.30}{C-2}$, all $f_{\ol{i}}(a_r, a_{r+1},\dots,a_{r+\ol{i}})$, for $i_j+1 \le r \le (1-\ep_j)\ep_1\dots\ep_{j-1}i$, are $0$ or $C-j$ corresponds exactly to: with probability at least $\frac{0.30}{C-2}$, a simple random walk in $[C]_0^{\ol{i}}$ of length $L := i_{j-1}-2i_j$ consists entirely of red vertices. 

\vs

Hence, by Proposition \ref{bootstrap}, with probability at least $\frac{1}{20C^2}$ a simple random walk of length\footnote{To be light on notation, we suppress ceiling and floor functions in the rest of this section.} $(1+\frac{1}{20C^2})L$ consists entirely of red vertices. Now, $(1+\frac{1}{20C^2})L \ge (1+\frac{1}{40C^2})i_{j-1}$ since it is equivalent to $\frac{1}{40C^2}i_{j-1} \ge (2+\frac{1}{10C^2})i_j$, which is true since $i_j \le \frac{i_{j-1}}{100C^2}$. We have thus shown that, if $a_1,\dots,a_{(1+\frac{1}{40C^2})i_{j-1}+\ol{i}}$ are chosen independently and uniformly at random from $[C]_0$, then with probability at least $\frac{1}{20C^2}$, all $f_{\ol{i}}(a_r,\dots,a_{r+\ol{i}})$ for $1 \le r \le (1+\frac{1}{40C^2})i_{j-1}$ are either $0$ or $C-j$. 

\vs

We're nearly done, as $(f_{\ol{i}}(a_r,\dots,a_{r+\ol{i}}))_{1 \le r \le L'}$ is the whole sequence after $\ol{i}$ iterations; since $C-j$ is odd, we just need to additionally ensure that the ultimate iterate is even. Specifically, we argue as follows. 

\vs 

We now deduce that, for $L' := i_{j-1}$, if $a_1,\dots,a_i$ are chosen independently and uniformly at random from $[C]_0$, then with probability at least $\frac{1}{160C^2}$, they ultimately iterate to something $0$ mod $2$ and each $f_{\ol{i}}(a_r,\dots,a_{r+\ol{i}})$, for $1 \le r \le L'$, are either $0$ or $C-j$. Let $\delta = \frac{1}{40C^2}$. By Lemma \ref{mod2}, the proportion of walks $(X_1,\dots,X_{(1+\delta)L'})$ in $[C]_0^{\ol{i}}$ of length $(1+\delta)L'$ that have at most $\frac{\delta L'}{4}$ values of $j \in [\delta L']$ with\footnote{Here we have abused notation, by associating the $i$-tuple that $X_{j+1},\dots,X_{j+L'}$ form with $(X_{j+1},\dots,X_{j+L'})$.} $(X_{j+1},X_{j+2},\dots,X_{j+L'}) \in \mc{B}_0$ is at most\footnote{The inequality following this footnote follows from the well known ${n \choose k} \le (\frac{en}{k})^k$, giving $\frac{\delta L'}{4}{\delta L' \choose \delta L'/4}2^{-\delta L'} \le \frac{\delta L'}{4}(\frac{e\delta L'}{\delta L'/4})^{\delta L'/4}2^{-\delta L'} < \frac{\delta L'}{4}(0.91)^{\delta L'}$. Note $\delta L' \ge \frac{1}{40C^2}(200C^2)^C$.} $\frac{\delta L'}{4}{\delta L' \choose \delta L'/4}2^{-\delta L'} \le \frac{1}{40C^2}$. Therefore, since the proportion of walks $(X_1,\dots,X_{(1+\delta)L'})$ with $X_1,\dots,X_{(1+\delta)L'}$ all red is at least $\frac{1}{20C^2}$, if we let $\mc{A}$ denote the walks $(X_1,\dots,X_{(1+\delta)L'})$ such that $X_1,\dots,X_{(1+\delta)L'}$ are all red and such that there are at least $\frac{\delta L'}{4}$ values of $j$ with $(X_{j+1},X_{j+2},\dots,X_{j+L'}) \in \mc{B}_0$, then the density of $\mc{A}$ is at least $\frac{1}{40C^2}$. So on one hand, $$\sum_{(X_1,\dots,X_{(1+\delta)L'}) \in \mc{A}} \hs \sum_{j=1}^{\delta L'} 1_{(X_{j+1},\dots,X_{j+L'}) \in \mc{B}_0} \ge \frac{\delta L'}{4}\frac{1}{40C^2}C^{\ol{i}}C^{(1+\delta)L'-1},$$ while on another hand, \begin{align*} \sum_{(X_1,\dots,X_{(1+\delta)L'}) \in \mc{A}} \hs \sum_{j=1}^{\delta L'} 1_{(X_{j+1},\dots,X_{j+L'}) \in \mc{B}_0} &= \sum_{j=1}^{\delta L'}\sum_{(X_{j+1},\dots,X_{j+L'}) \in \mc{B}_0} \sum_{\substack{X_1,\dots,X_j,X_{j+L'+1},\dots,X_{(1+\delta)L'} \\ (X_1,\dots,X_{(1+\delta)L'}) \in \mc{A}}} 1 \\ & \le \sum_{j=1}^{\delta L'}\sum_{(X_{j+1},\dots,X_{j+L'}) \in \mc{B}_0} C^{\delta L'} 1_{X_{j+1},\dots,X_{j+L'} \text{ all red}} \\ &= \delta L' C^{\delta L'} \sum_{(X_1,\dots,X_{L'}) \in \mc{B}_0} 1_{X_1,\dots,X_{L'} \text{ all red}}.\end{align*} We deduce that $$\sum_{(X_1,\dots,X_{L'}) \in \mc{B}_0} 1_{X_l,\dots, X_{L'} \text{ all red}} \ge \frac{1}{160C^2}C^{\ol{i}}C^{L'-1},$$ which is what we wanted to deduce. 
\end{proof}

\vs

\begin{corollary}\label{0bound}
For any $C \ge 2$ and any $i \ge 1$, if $a_1,\dots,a_i$ are chosen independently and uniformly at random from $\{0,\dots,C-1\}$, then the probability they ultimately iterate to $0$ is at least $(\frac{1}{C})^{(200C^2)^{2C}}$.
\end{corollary}

\begin{proof}
For $i \ge (200C^2)^{2C}$, Proposition \ref{lowerbound} yields a lower bound of $\frac{1}{200C^2}$, and for $1 \le i < (200C^2)^{2C}$, we use the trivial lower bound coming from $a_j = 0$ for all $j$.
\end{proof}

\section{Finishing the Proof of Theorem \ref{main}}

We now finish the proof of Theorem \ref{main}, copied below for the reader's convenience. 

\setcounter{theorem}{1}
\begin{theorem}\label{main}
For $M$ large, for any $C$ with $2 \le C \le \frac{1}{100}\frac{\log\log M}{\log\log\log M}$, if we form an initial sequence of length $M$ by choosing numbers from $\{0,\dots,C-1\}$ independently and uniformly at random, then, with probability at least $1-e^{-e^{\sqrt[20]{\log M}}}$, after $e^{\sqrt[5]{\log M}}$ iterations of consecutive differencing, everything is a $0$ or $1$. 
\end{theorem}

\vspace{1mm}

Fix $M$ large and $C$ in the range $[3,\frac{1}{100}\frac{\log\log M}{\log\log\log M}]$ (the case $C=2$ is trivial). Let $E_1$ denote\footnote{To be light on notation, we suppress ceiling and floor functions in this section.} the event that after $0$ iterations, there is a $\{0,C-1\}$-block of length $R := e^{\sqrt[10]{\log M}}$. Let $E_2$ be the event that after $2R$ iterations, there is a $\{0,C-2\}$-block of length $R^2$. Let $E_3$ be the event that after $2R^2$ iterations, there is a $\{0,C-3\}$-block of length $R^3$. In general, for $2 \le j \le C-2$, $E_j$ is the event that after $2R^{j-1}$ iterations, there is a $\{0,C-j\}$-block of length $R^j$. Since $2R^{j-1} \ge 2R^{j-2}+R^{j-1}$ for $3 \le j \le C-1$, we see that, as before, by Lemma \ref{blockdestruction}, if no $E_j$ occurs, then after $2R^{C-2}$ iterations, everything is a $0$ or a $1$. Note that $2R^{C-2} \le e^{\sqrt[5]{\log M}}$, so it suffices to show that the probability that some $E_j$ occurs is at most $e^{-e^{\sqrt[20]{\log M}}}$. By the union bound, it suffices to show $\Pr(E_j) \le e^{-e^{\sqrt[13]{\log M}}}$, say, for each $1 \le j \le C-2$. 

\vs

Clearly, $\Pr(E_1) \le M(\frac{2}{3})^R \le e^{-e^{\sqrt[13]{\log M}}}$, so fix some $j$ with $2 \le j \le C-2$. By Lemma \ref{inverseiterates}, if $E_j$ occurs, either there is a $(C-j)\Z$-block of length $R^j$ in the initial sequence or there is a block of length $R^j$ in the first $2R^{j-1}-1$ iterations containing no $0$s. Once again, the first option holds with probability at most $M(\frac{2}{3})^{R^j} \le \frac{1}{2}e^{-e^{\sqrt[13]{\log M}}}$, so by the union bound, it suffices to show that for each $0 \le i \le 2R^{j-1}-1$, the probability that there is a block of length $L := R^j = e^{j\sqrt[10]{\log M}}$ without $0$s after $i$ iterations is at most $e^{-e^{\sqrt[12]{\log M}}}$, say. 

\vs

So fix some $i \in [0,2R^{j-1}-1]$. Let $b_1,\dots,b_{M-i}$ denote the sequence after $i$ iterations. Let's first focus on the block $b_1,\dots,b_L$. Say the initial sequence is $a_1,\dots,a_M$. Note that $b_{k(i+1)+1} = f_i(a_{k(i+1)+1},\dots,a_{(k+1)(i+1)})$ for $0 \le k \le \frac{1}{2}R-1$. Since $(\frac{1}{2}R-1)(i+1)+1 \le \frac{1}{2}R(i+1) \le L$ and the sets $\{a_{k(i+1)+1},\dots,a_{(k+1)(i+1)}\}$ are disjoint as $k$ ranges, by independence the probability that $b_1,\dots,b_L$ are all nonzero is at most $\left(1-(\frac{1}{C})^{(200C^2)^{2C}}\right)^{R/2}$ by Corollary \ref{0bound}. Using the standard $1-x \le e^{-x}$, we see  that $\left(1-(\frac{1}{C})^{(200C^2)^{2C}}\right)^{R/2} \le \exp\left(-\frac{R}{2}(\frac{1}{C})^{(200C^2)^{2C}}\right) \le \exp\left(-\frac{R}{2}e^{-(\log C)e^{5C\log C}}\right) \le \exp\left(-\frac{R}{2}e^{-(\log\log\log M)e^{\frac{1}{19}\log\log M}}\right) \le \exp\left(-\frac{R}{2}e^{-\sqrt[15]{\log M}}\right) \le \exp\left(-e^{\sqrt[11]{\log M}}\right)$. Therefore, by the union bound, the probability that there is some block of length $L$ after $i$ iterations containing no $0$s is at most $Me^{-e^{\sqrt[11]{\log M}}} \le e^{-e^{\sqrt[12]{\log M}}}$. The proof is thus complete. \qed

\section{Proof of Theorem \ref{truemain}}

In this section we deduce Theorem \ref{truemain} from Theorem \ref{main}. We start with a lemma. 

\vspace{1mm}

\begin{lemma}\label{truncated}
Take $M$ large. Let $f: [M] \to \{2,3,\dots,\lfloor \frac{1}{100}\frac{\log\log M}{\log\log\log M}\rfloor\}$ be an increasing function. Form a random initial sequence $b_1,\dots,b_M$ by choosing $b_m$ uniformly at random from $\{0,1,\dots,f(n)-1\}$, independently of the other $b_i$'s. Then, with probability at least $1-e^{-\frac{1}{20}\log^2 M}$, after $3\frac{M}{\log^2 M}$ iterations of consecutive differencing, everything is a $0$ or $1$. 
\end{lemma}

\vspace{1.5mm}

Before proving Lemma \ref{truncated}, let's prove Theorem \ref{truemain} assuming it. 

\vspace{1.5mm}

\begin{proof}[Proof of Theorem \ref{truemain}]
Let $A_M$ denote the event that after $M$ iterations, the first term is not a $1$. We wish to show that, with probability $1$, only finitely many $A_M$'s occur. By Borel-Cantelli, it suffices to show that for all $M$ large, the probability of $A_M$ occurring is at most $e^{-\frac{1}{30}\log^2 M}$. Note that $A_M$ is equivalent to $a_1,\dots,a_{M+1}$ not ultimately iterating to $1$. For $M$ large enough, by Lemma \ref{truncated}, with probability at least $1-e^{-\frac{1}{20}\log^2 M}$, after $3\frac{M}{\log^2 M}$ iterations of consecutive differencing beginning with initial sequence $u_2,\dots,u_M$, everything is a $0$ or $1$. Therefore, with probability at least $1-e^{-\frac{1}{20}\log^2 M}$, after $3\frac{M}{\log^2 M}$ iterations of consecutive differencing beginning with initial sequence $2u_2,\dots,2u_M$, everything is a $0$ or $2$. It follows that with probability at least $1-e^{-\frac{1}{20}\log^2 M}$, after $1+3\frac{M}{\log^2 M}$ iterations of consecutive differencing beginning with initial sequence $a_1,\dots,a_{M+1}$, the obtained sequence starts off with an odd number at most $\frac{1}{100}\frac{\log\log M}{\log\log\log M}$ followed by only $0$s and $2$s. By Lemma \ref{mod2}, with probability at least $1-e^{-\frac{1}{10}\log^2M}$, the second term of the sequence is congruent to $2 \modd 4$ at least $\frac{1}{3}\log^2 M$ times out of the $\log^2 M$ iterations following the $(1+3\frac{M}{\log^2 M})^{\text{th}}$ iteration. Therefore, with probability at least $1-e^{-\frac{1}{20}\log^2 M}-e^{-\frac{1}{10}\log^2 M} \ge 1-e^{-\frac{1}{30}\log^2 M}$, starting with $a_1,\dots,a_{M+1}$, after $1+3\frac{M}{\log^2 M}+\log^2 M$ iterations, the first term will be a $1$, and therefore will remain a $1$ all the way until the final (i.e., $M^{\text{th}}$) iteration, since everything else is a $0$ or $2$. 
\end{proof}

\vspace{1mm}

\begin{definition}
Let $a_1,\dots,a_{M+1}$ be non-negative integers. We say that an index $i \in [M+1]$ \textit{influenced} the index $j \in [M+1-t]$ after $t$ iterations if $0 \le i-j \le t$. Recall that $f_t(a_j,\dots,a_{j+t})$ is the value at index $j$ after $t$ iterations. 
\end{definition}

\vspace{1.5mm}

We finish by proving Lemma \ref{truncated}. The idea of the proof is as follows. By Theorem \ref{main}, the blocks on which $f$ is constant will become all $0$s and $1$s after not too many iterations. Although there are some indices that were influenced by indices where $f$ took different values, these indices are contained in not too many not too large intervals, so we can let all the $0$s and $1$s drop the values at these ``bad indices" with a few extra iterations.

\vs

We start by proving a lemma that allows us to isolate these ``bad indices". For an interval $I \sub \N$, let $L(I)$ and $R(I)$ denote its left and right endpoints, respectively. 

\vspace{1.5mm}

\begin{lemma}\label{intervals}
Suppose $M$ is large, and let $C_M$ be a positive integer with $C_M \le \log\log M$. Let $I_1,\dots,I_r \sub [M]$ be disjoint intervals with $r \le C_M$ and $|I_t| \le C_Me^{\sqrt[5]{\log M}}$ for each $t$. Then there are pairwise disjoint intervals $J_1,\dots,J_s \sub [M]$, each containing some $I_t$, such that the following two hold. \begin{itemize} \item For all $t$, $1 \le t \le r$, there is some $m$ with $I_t \sub J_m$. \item For any $m$, $1 \le m \le s$, if we let $B_m$ denote the smallest interval containing all of the $I_t$'s in $J_m$, then we have that either $L(B_m)-L(J_m) \ge (\log^2 M)^{C_M}|B_m|$ or $R(J_m)-R(B_m) \ge (\log^2 M)^{C_M}|B_m|$, with both being true if $J_m$ contains neither $1$ nor $M$.\end{itemize}
\end{lemma}

\begin{proof}
For a subset $A$ of $[r]$, let $B_A$ denote the smallest interval containing $\cup_{t \in A} I_t$, and let $J(A)$ denote the smallest interval containing $\cup_{t \in A} I_t$ such that either $L(B_A)-L(J(A)) \ge (\log^2 M)^{C_M}|B_A|$ or $R(J(A))-R(B_A) \ge  (\log^2 M)^{C_M}|B_A|$, with both being true if $J(A)$ contains neither $1$ nor $M$; if no such interval exists, we let $J(A) = \emptyset$. Let $\mathcal{C}_0 = \{J(\{t\}) : 1 \le t \le r\}$. For $i \ge 0$, if $\mathcal{C}_i$ contains two intervals $J(A_1),J(A_2)$ that intersect, we define $\mathcal{C}_{i+1}$ to be the same as $\mathcal{C}_i$, except we replace $J(A_1)$ and $J(A_2)$ with $J(A_1\cup A_2)$ ($\mathcal{C}_{i+1}$ thus could depend on the choice of intersecting intervals). Say $\mathcal{C}_0,\dots,\mathcal{C}_{k-1}$ are the defined collections. It is clear that $k \le r$ and that if each element of $\mathcal{C}_{k-1}$ is non-empty, then the elements of $\mathcal{C}_{k-1}$ satisfy the conditions of Lemma \ref{intervals}. The largest diameter of an interval in $\mathcal{C}_0$ is at most $(2(\log^2 M)^{C_M}+1)C_Me^{\sqrt[5]{\log M}} \le 3(\log^2 M)^{C_M}C_Me^{\sqrt[5]{\log M}}$. If $J(A_1)$ and $J(A_2)$ each have diameter at most $D$ and intersect, then the diameter of $J(A_1\cup A_2)$ is at most $(2(\log^2 M)^{C_M}+1)(2D) \le 6(\log^2 M)^{C_M}D$. Therefore, each interval in any $\mathcal{C}_{i-1}$ has diameter at most $6^{i-1}(\log^2 M)^{(i-1)C_M}3(\log^2 M)^{C_M}C_Me^{\sqrt[5]{\log M}} \le 6^r(\log^2 M)^{rC_M} C_Me^{\sqrt[5]{\log M}} \le e^{\sqrt[4]{\log M}}$. To finish the proof, it just remains to note that $J(A) \not = \emptyset$ if the diameter of $\cup_{t \in A} I_t$ is at most $e^{\sqrt[4]{\log M}}$. 
\end{proof}

\vs

\begin{proof}[Proof of Lemma \ref{truncated}]
Do $e^{\sqrt[5]{\log M}}$ iterations of consecutive differencing. For $2 \le C \le \frac{1}{100}\frac{\log\log M}{\log\log\log M} =: C_M$, we say that an index $j$ is \textit{$C$-pure} if $f$ took the value $C$ at all indices in the initial sequence that influenced $j$ (after $e^{\sqrt[5]{\log M}}$ iterations). Let $I$ denote the indices that are not $C$-pure for any $C$. Write $I = \sqcup_{t=1}^r I_t$ as a disjoint union of intervals with $r$ minimal. Clearly $r \le C_M$. Also, crudely, $|I_t| \le C_Me^{\sqrt[5]{\log M}}$ for each $t$. 

\vs

Let $J_1,\dots,J_s$ be the intervals guaranteed\footnote{We are applying Lemma \ref{intervals} with $M-e^{\sqrt[5]{\log M}}$ instead of $M$, but all bounds are essentially the same.} by Lemma \ref{intervals}, and let $B_1,\dots,B_s$ be as in Lemma \ref{intervals}. For any $C$, by\footnote{As stated, Theorem \ref{main} only applies to initial sequences of length $M$. However, given any shorter initial sequence, we can independently add elements uniformly chosen from $\{0,\dots,C-1\}$ to obtain a sequence of length $M$, then do $e^{\sqrt[5]{\log M}}$ iterations, and then truncate the sequence to keep only indices influenced by the original initial sequence.} Theorem \ref{main} applied to the (interval of) $C$-pure indices, the probability that all $C$-pure indices are $0$ or $1$ is at least $1-e^{-e^{\sqrt[20]{\log M}}}$, and therefore the probability that all indices that are $C$-pure for some $C$ are $0$ or $1$ is at least $1-C_Me^{-e^{\sqrt[20]{\log M}}} \ge 1-e^{-\sqrt[21]{\log M}}$. In particular, with probability at least $1-e^{-\sqrt[21]{\log M}}$, all indices in $\cup_{m=1}^s (J_m\setminus B_m)$ are $0$ or $1$; we from here on condition on this being the case. For $1 \le m \le s$ and $1 \le j \le C_M-1$, let $J_m^j$ denote the interval (of length $|J_m|-2(\log^2 M)^j |B_m|$) whose indices after $2(\log^2 M)^j |B_m|$ iterations past the $e^{\sqrt[5]{\log M}}$th are influenced by indices only in $J_m$, and let $B_m^j$ denote the interval (of length $|B_m|+2(\log^2 M)^j |B_m|$) whose indices after $2(\log^2 M)^j |B_m|$ iterations past the $e^{\sqrt[5]{\log M}}$th are influenced by at least one index in $B_m$. Note that Lemma \ref{intervals} implies $B_m^j \subseteq J_m^j$ for each $1 \le j \le C_M-1$ (since $2(\log^2 M)^{C_M-1}|B_m| \le (\log^2 M)^{C_M}|B_m|$). 

\vs

For $1 \le m \le s$, let $E_m^0$ denote the event that there is a $\{0,C_M\}$-block in $J_m$ of length $(\log^2 M)|B_m|$ containing a $C_M$. For $1 \le m \le s$ and $1 \le j \le C_M-2$, let $E_m^j$ denote the event that, after $2(\log^2 M)^j |B_m|$ iterations (past the $e^{\sqrt[5]{\log M}}$th), there is a $\{0,C_M-j\}$-block in $J_m^j$ of length $(\log^2M)^{j+1}|B_m|$ containing a $C_M-j$. Fix $m$ with $1 \le m \le s$. As in the proofs of Proposition \ref{lowerbound} and Theorem \ref{main}, since $2(\log^2 M)^{i+1} |B_m| \ge (\log^2 M)^{i+1}|B_m|+2(\log^2 M)^i|B_m|$, if none of $E_m^0, E_m^1,\dots,E_m^{C_M-2}$ occur, then after $2(\log^2 M)^{C_M-1}$ iterations, the largest number in $J_m^{C_M-1}$ is a $1$. 

\vs

Note that any $C_M$'s in $J_m$ lie in $B_m$, so by Lemma \ref{mod2}, the probability that $E_m^0$ occurs is at most $2(\frac{1}{2})^{\frac{1}{2}\log^2 M}$, since either to the left or to the right of $B_m$ must be $\frac{1}{2}\log^2M$ consecutive $0$s. Similarly, the length of the longest $\{0,C_M-j\}$-block in $J_m^j$ is at most the whole of $B_m^j$ and $0$s surrounding it, so the probability $E_m^j$ occurs is at most $2(\frac{1}{2})^{\frac{1}{4}\log^2 M}$. Therefore, the probability that at least one of $E_m^0,\dots,E_m^{C_M-2}$ occurs is at most $2(\frac{1}{2})^{\frac{1}{2}\log^2 M}+(C_M-2)2(\frac{1}{2})^{\frac{1}{4}\log^2 M} \le e^{-\frac{1}{10}\log^2 M}$. Since $B_m^{C_M-1} \subseteq J_m^{C_M-1}$, if none of $E_m^0,\dots,E_m^{C_m-2}$ occur, then the elements of (the growing) $B_m$ became $0$ and $1$ quickly enough to not affect anything outside of (the shrinking) $J_m$. In particular, if none of $E_m^0,\dots,E_m^{C_M-2}$ occur for any $m$ (i.e. for each $m$, none occur), then\footnote{It is clear from Lemma \ref{intervals} that $|B_m| \le \frac{M}{(\log^2 M)^{C_M}}$ for each $m$.} after $2(\log^2 M)^{C_M-1}\max_{1 \le m \le s} |B_m| \le 2\frac{M}{\log^2 M}$ iterations past the $e^{\sqrt[5]{\log M}}$th, everything is a $0$ or $1$. Since the probability at least one $E_m^j$ (over all $j,m$) occurs is at most $se^{-\frac{1}{10}\log^2 M} \le e^{-\frac{1}{20}\log^2 M}$, Lemma \ref{truncated} is established. 
\end{proof}

\section{Additional Mathematical Remarks}

The proof of Theorem \ref{main} can be relatively easily adapted to handle any distribution (not just the uniform distribution) on $\{0,\dots,C-1\}$ that gives not too large, positive weight to each of $0,\dots,C-1$ (one should create duplicate vertices in $[C]_0^i$ so that the obtained simple random walk models this different probability distribution). 

\vs

In Theorem \ref{main} we did not try to optimize $e^{-e^{\sqrt[20]{\log M}}}$ nor $e^{\sqrt[5]{\log M}}$. A proof allowing $C$ to go all the way up to $\log^2 M$, or even a power of $M$, would be interesting. We expect that, in reality, the highest $C$ can go is $M$, in that if $C = o(M)$, then with probability $1-o(1)$, after $\frac{M}{2}$ iterations, everything is a $0$ or $1$, while if $C = \omega(M)$, with probability $o(1)$, after $\frac{M}{2}$ iterations, everything is a $0$ or $1$. 

\vs

\section{A Historical Remark}\label{history}

Various sources (websites, blog posts, etc.) have claimed that Proth believed he had proven Gilbreath's conjecture, and that his proof turned out to be wrong. 

\vs

Not only do we currently have no evidence for this claim, the apparent source of this claim has retracted it. 

\vs

The claim seemed plausible, for Proth did publish a paper \cite{proth2} on (what later became known as) Gilbreath's conjecture and did, admittedly confusingly, call it a ``theorem". However, a reading through the paper shows he did not seriously claim a proof. Indeed, Hugh Williams who made the claim about Proth without reference \cite[p.~123]{williams}, said ``On rereading his actual paper ... I can find no support for my assertion. ... My apologies for seeming to have started a myth" \cite{hughwilliams}.

\vs

We also take this time to correct another historical error, which actually is composed of two suberrors. The first suberror is that many sources incorrectly cited \cite{proth1} when referring to Proth's discussion of Gilbreath's conjecture, referring to the correct title ``Th\'{e}or\`{e}mes sur les nombres premiers" but citing Comp. Rend. Acad. Sci. Paris, 85 (1877) instead of Comp. Rend. Acad. Sci. Paris, 87 (1877). The former actually corresponds to a completely unrelated paper of Pepin \cite{pepin}. The second suberror is that, the intended reference, \cite{proth1}, didn't even discuss Gilbreath's conjecture! We were only able to find Proth discussing Gilbreath's conjecture in \cite{proth2}.

\vs

We refer the reader to \cite{juanarias} for more information surrounding all of this.

\vs

\section{Acknowledgments} 

I would like to thank my advisor, Ben Green, for suggesting this problem to me and Daniel Korandi for helpful feedback on the introduction. I would also like to thank Juan Arias de Reyna for bringing to attention the dubious nature of the claim discussed in Section \ref{history}, and Hugh Williams for kindly responding to emails and helping resolve the situation.

\vs

\end{document}